\documentclass{birkjour_t2}
\usepackage{amsmath}
\usepackage{amsfonts}
\usepackage{amssymb}
\usepackage[english]{babel}
\usepackage{color}
\usepackage{graphicx}
\usepackage{lmodern}
\usepackage{amsthm}
\usepackage{mathrsfs}
\usepackage{microtype}
\usepackage{mathscinet}
\usepackage{enumitem}
\usepackage[cal=boondoxo,bb=ams]{mathalfa}
\usepackage{hyperref}
\hypersetup{hidelinks}

\renewcommand{\theequation}{\arabic{section}.\arabic{equation}}

\newtheorem{thm}{Theorem}[section]
\newtheorem{prop}{Proposition}[section]
\newtheorem{lem}{Lemma}[section]

\newtheorem{rem}{Remark}[section]

\newcommand{\ml}{\mathcal}
\newcommand{\mb}{\mathbb}
\DeclareMathOperator{\diag}{diag}
\DeclareMathOperator{\divv}{div}
\DeclareMathOperator{\intt}{int}
\DeclareMathOperator{\extt}{ext}

\begin{document}

%
%
%
%
%
%
%
%

\title[Doubly dissipative elastic waves]
 {Dissipative structure and diffusion phenomena for doubly dissipative elastic waves in two space dimensions}

\author[W. Chen]{Wenhui Chen}
\address{Institute of Applied Analysis, Faculty for Mathematics and Computer Science\\
	 Technical University Bergakademie Freiberg\\
	  Pr\"{u}ferstra{\ss}e 9\\
	   09596 Freiberg\\
	    Germany}
\email{wenhui.chen.math@gmail.com}


\date{January 1, 2004}

\begin{abstract} In this paper we study the Cauchy problem for doubly dissipative elastic waves in two space dimensions, where the damping terms consist of two different friction or structural damping. We derive energy estimates and diffusion phenomena with different assumptions on initial data. Particularly, we find the dominant influence on diffusion phenomena by introducing a new threshold of diffusion structure.

\end{abstract}

\maketitle

\section{Introduction}\label{Introduction}
In this paper we consider the following Cauchy problem for doubly dissipative elastic waves in two space dimensions:
	\begin{equation}\label{Eq.DoublyDissElasticWaves}
	\left\{
\begin{aligned}
&u_{tt}-a^2\Delta u-\left(b^2-a^2\right)\nabla\divv u+(-\Delta)^{\rho}u_t+(-\Delta)^{\theta}u_t=0,&&x\in\mb{R}^2,\,\,t>0,\\
&(u,u_t)(0,x)=(u_0,u_1)(x),&&x\in\mb{R}^2,
\end{aligned}
\right.
\end{equation}
where the unknown $u=u(t,x)\in\mb{R}^2$ denotes the elastic displacement. The positive constants  $a$ and $b$ in \eqref{Eq.DoublyDissElasticWaves} are related to the Lam\'e constants and fulfill $b>a>0$. Moreover, the parameters $\rho$ and $\theta$ in \eqref{Eq.DoublyDissElasticWaves} satisfy $0\leq\rho<1/2<\theta\leq1$.

Let us recall some related works to our problem \eqref{Eq.DoublyDissElasticWaves}. Taking $a=b=1$, $\rho=0$ and $\theta=1$ in \eqref{Eq.DoublyDissElasticWaves}, then we immediately turn to doubly dissipative wave equation, where the damping terms consist of \emph{friction} $u_t$ as well as \emph{viscoelastic damping} $-\Delta u_t$
\begin{equation}\label{Eq.DoublyDissWave}
\left\{
\begin{aligned}
&u_{tt}-\Delta u+u_t-\Delta u_t=0,&&x\in\mb{R}^n,\,\,t>0,\\
&(u,u_t)(0,x)=(u_0,u_1)(x),&&x\in\mb{R}^n,
\end{aligned}
\right.
\end{equation}
with $n\geq1$. The recent paper \cite{IkehataSawada2016} derived asymptotic profiles of solutions to \eqref{Eq.DoublyDissWave} in a framework of weighted $L^1$ data. Precisely, the authors found that from asymptotic profiles of solutions point of view,  friction $u_t$  is more dominant than viscoelastic damping $-\Delta u_t$ as $t\rightarrow\infty$. Later in \cite{IkehataMichihisa2018}, the authors obtained higher-order asymptotic expansions of solutions to \eqref{Eq.DoublyDissWave} and gave some lower bounds estimates to show the optimality of these expansions. For the other related works on \eqref{Eq.DoublyDissWave}, we refer the reader to the recent papers \cite{IkehataTakeda2017,DAbbicco2017,IkehataTakeda2019}. However, asymptotic profiles of solutions to general doubly dissipative wave equation, where the damping terms consist of friction or structural damping (i.e., taking $a=b=1$ in \eqref{Eq.DoublyDissElasticWaves}), are still open. This open problem is proposed in \cite{IkehataSawada2016}. The main difficulty is to answer what is the dominant profile of solutions, due to the fact that the asymptotic profiles for wave equation with damping term $(-\Delta)^{\rho}u_t$ for $0\leq \rho<1/2$, or with damping term $(-\Delta)^{\theta}u_t$ for $1/2<\theta\leq 1$, are quite different. One may see, for example, \cite{Matsumura1976,Karch2000,MarcatiNishihara2003,HosonoOgawa2004,Narazaki2004,Nishihara2003,Takeda2015,Ikehata2014,DabbiccoEbert2014,IkehataOnodera2017,Michihisa2017,Michihisa2018,IkehataTakeda2019NEW,Shibata2000,Ponce1985,DabbiccoReissig2014}.

 Let us come back to dissipative elastic waves. In recent years the Cauchy problem for dissipative elastic waves have aroused wide concern, which can be modeled by
	\begin{equation}\label{Eq.DissElasticWaves}
\left\{
\begin{aligned}
&u_{tt}-a^2\Delta u-\left(b^2-a^2\right)\nabla\divv u+\ml{A}u_t=0,&&x\in\mb{R}^n,\,\,t>0,\\
&(u,u_t)(0,x)=(u_0,u_1)(x),&&x\in\mb{R}^n,
\end{aligned}
\right.
\end{equation}
where $b>a>0$ and the term $\ml{A}u_t$ describes several kinds of damping mechanisms.\\
 In the case when
\begin{align*}
\ml{A}u_t=u_t,\,\,\,\,\text{i.e., \emph{friction} or \emph{external damping}},
\end{align*}
 the authors of \cite{IkehataCharaodaLuz2014} proved almost sharp energy estimates for $n\geq2$ by using energy methods in the Fourier space and the Haraux-Komornik inequality, and then the recent paper \cite{ChenReissig2019SD} investigated propagation of singularities, sharp energy estimates and diffusion phenomenon for $n=3$.\\
  Furthermore, in the case when
 \begin{align*}
 \ml{A}u_t=(-\Delta)^{\theta}u_t\,\,\,\,\text{with}\,\,\,\,\theta\in(0,1],\,\,\,\,\text{i.e., \emph{structural damping}},
 \end{align*}
    energy estimates are derived with different data spaces in \cite{IkehataCharaodaLuz2014} for $n\geq2$, and in \cite{Reissig2016} for $n=2$. Moreover, some qualitative properties of solutions, including smoothing effect, sharp energy estimate and diffusion phenomena (especially, \emph{double diffusion phenomena} when $\theta\in(0,1/2)$) are obtained for $n=3$.\\
    Finally, in the case when 
    \begin{align*}
    \ml{A}u_t=(-a^2\Delta-(b^2-a^2)\nabla\divv)u_t,\,\,\,\,\text{i.e., \emph{Kelvin-Voigt damping}},
    \end{align*}
     by applying energy methods in the Fourier space, almost sharp energy estimates for $n\geq2$ have been obtained in \cite{WuChaiLi2017}. Then, sharp energy estimates, $L^p-L^q$ estimates as well as asymptotic profiles of solutions are derived for $n=2$ in \cite{Chen2019KV}. Other studies on dissipative elastic waves can be found in literatures \cite{CharaoIkehata2007,CharaoIkehata2011}. Nevertheless, concerning about decay properties and diffusion phenomena for the Cauchy problem for doubly dissipative elastic waves it seems that we still do not have any previous research manuscripts. Moreover, this problem is strongly related to the open problem proposed in \cite{IkehataSawada2016}. In this paper we give the answer to the two-dimensional case. 
     
     Let us point out that the study of the Cauchy problem \eqref{Eq.DoublyDissElasticWaves} is not simply a generalization of elastic waves with friction or structural damping in \cite{Reissig2016,ChenReissig2019SD}. On one hand, because there exists two different damping terms $(-\Delta)^{\rho}u_t$ and $(-\Delta)^{\theta}u_t$ with $0\leq\rho<1/2<\theta\leq1$ in our problem \eqref{Eq.DoublyDissElasticWaves}, it is not clear which damping term has a dominant influence on dissipative structure. On the other hand, from the paper \cite{ChenReissig2019SD}, the authors derived diffusion phenomena to elastic waves with the damping term $(-\Delta)^{\theta}u_t$ where $\theta\in[0,1/2)\cup(1/2,1]$, which are described by the following so-called \emph{reference system}.
\begin{itemize}
	\item In the case when $\theta=0$, the reference system consist of heat-type system with mass term as follows:
	\begin{align*}
	\widetilde{V}_t-\ml{D}_1\Delta\widetilde{V}+\ml{D}_2\widetilde{V}=0,
	\end{align*}
	with real diagonal matrices $\ml{D}_1$ and $\ml{D}_2$.
	\item In the case when $\theta\in(0,1/2)$, the reference system consist of two different parabolic systems as follows:
	\begin{align*}
	\widetilde{V}_t+\ml{D}_3(-\Delta)^{1-\theta}\widetilde{V}+\ml{D}_4(-\Delta)^{\theta}\widetilde{V}=0,
	\end{align*}
	with real diagonal matrices $\ml{D}_3$ and $\ml{D}_4$.
	\item In the case when $\theta\in(1/2,1]$, the reference system consist of parabolic system and half-wave system as follows:
	\begin{align*}
	\widetilde{V}_t+\ml{D}_5(-\Delta)^{\theta}\widetilde{V}+i\ml{D}_6(-\Delta)^{\frac{1}{2}}\widetilde{V}=0,
	\end{align*}
	with real diagonal matrices $\ml{D}_5$ and $\ml{D}_6$.
\end{itemize}
Hence, for different choices of damping terms, which mainly depend on the value of the parameter $\theta$ in the damping term, the diffusion phenomena are quite different. In the Cauchy problem \eqref{Eq.DoublyDissElasticWaves}, the damping terms consist of $(-\Delta)^{\rho}u_t$ with $\rho\in[0,1/2)$, and $(-\Delta)^{\theta}u_t$ with $\theta\in(1/2,1]$. Thus, it is not clear that the reference system is make up of what kind of evolution systems, and how do two different damping terms influence on diffusion structure. Furthermore, from \cite{ChenReissig2019SD} we know \emph{the threshold of diffusion structure} is $\theta=1/2$ for elastic waves with structural damping. In other words, the structure of reference system will be changed from $\theta\in(0,1/2)$ to $\theta\in(1/2,1]$.  Then, the natural question is what is the threshold of diffusion structure for doubly dissipative elastic waves. Again, we give the answers for these questions in two dimensions.

Our main purpose of the present paper is to investigate dissipative structure and diffusion phenomena for doubly dissipative elastic waves with different assumptions on initial data. We find that the damping term $(-\Delta)^{\rho}u_t$ with $0\leq\rho<1/2$ has the dominant influence on energy estimates (see Theorems \ref{Thm.EnergyEst} and \ref{Thm.EnergyWeighted.L1}). Furthermore, in the case when $\rho+\theta<1$, the damping terms $(-\Delta)^{\rho}u_t$ and $(-\Delta)^{\theta}u_t$ with $0\leq\rho<1/2<\theta\leq1$ have the influence on diffusion structure at the same time. However, in the case when $\rho+\theta\geq1$, the diffusion structure is determined by the damping term $(-\Delta)^{\rho}u_t$ with $0\leq\rho<1/2$ only. Hence, one of our novelties is to derive a threshold $\rho+\theta=1$ of diffusion structure for doubly dissipative elastic waves.

This paper is organized as follows. In Section \ref{Sec.AsymptoticBehavior} we derive representation of solutions by applying WKB analysis and multistep diagonalization procedure. In Section \ref{Sec.EnergyEstimate} we obtain pointwise estimate in the Fourier space and energy estimates by using this representation. In Section \ref{Sec.AsymptoticProfiles} we derive diffusion phenomena with different assumptions on initial data. Finally, in Section \ref{Sec.ConcludingRemarks} some concluding remarks complete the paper.

\bigskip
\noindent\textbf{Notations: } In this paper $f\lesssim g$ means that there exists a positive constant $C$ such that $f\leq Cg$. We write $f\asymp g$ when $g\lesssim f\lesssim g$ Moreover, $H^s$ and $\dot{H}^s$ with $s\geq0$, denote Bessel and Riesz potential spaces based on $L^2$, respectively. Furthermore, $\langle D\rangle^s$ and $|D|^s$ stand for the pseudo-differential operators with symbols $\langle\xi\rangle^s$ and $|\xi|^s$, respectively, where $\langle\xi\rangle^2:=1+|\xi|^2$. We denote the identity matrix of dimensions $k\times k$ by $I_{k\times k}$. We denote the diagonal matrix by
\begin{align*}
\diag\left(e^{-\lambda_jt}\right)_{j=1}^4:=\diag\left(e^{-\lambda_1t},e^{-\lambda_2t},e^{-\lambda_3t},e^{-\lambda_4t}\right).
\end{align*}
The weighted spaces $L^{1,\gamma}$ for $\gamma\in[0,\infty)$ are defined by
\begin{align*}
L^{1,\gamma}:=\left\{f\in L^1:\|f\|_{L^{1,\gamma}}:=\int_{\mb{R}^n}(1+|x|)^{\gamma}|f(x)|dx<\infty\right\}.
\end{align*}
Finally, let us define the cut-off functions $\chi_{\intt}(\xi),\chi_{\text{bdd}}(\xi),\chi_{\extt}(\xi)\in \mathcal{C}^{\infty}$ having their supports in the following zones:
\begin{align*}
\ml{Z}_{\intt}(\varepsilon)&:=\left\{\xi\in\mb{R}^2:|\xi|<\varepsilon\ll1\right\},\\
\ml{Z}_{\text{bdd}}(\varepsilon,N)&:=\left\{\xi\in\mb{R}^2:\varepsilon\leq |\xi|\leq N\right\},\\
\ml{Z}_{\extt}(N)&:=\left\{\xi\in\mb{R}^2:|\xi|> N\gg1\right\},
\end{align*}
respectively, so that $\chi_{\intt}(\xi)+\chi_{\text{bdd}}(\xi)+\chi_{\extt}(\xi)=1$.

\section{Asymptotic behavior of solutions in the Fourier space}\label{Sec.AsymptoticBehavior}
In this section we will derive asymptotic behavior of solutions and representation of solutions in the Fourier space. Let us apply the partial Fourier transform with respect to spatial variable such that $\hat{u}(t,\xi)=\ml{F}_{x\rightarrow \xi}(u(t,x))$ to obtain
\begin{equation}\label{Eq.FourierDoublyDiss}
\left\{
\begin{aligned}
&\hat{u}_{tt}+|\xi|^{2}A(\eta)\hat{u}+\left(|\xi|^{2\rho}+|\xi|^{2\theta}\right)\hat{u}_t=0,&&\xi\in\mb{R}^2,
\,\,t>0,\\
&(\hat{u},\hat{u}_t)(0,\xi)=(\hat{u}_0,\hat{u}_1)(\xi),&&\xi\in\mb{R}^2,
\end{aligned}\right.
\end{equation}
where 
\begin{align*}
A(\eta)=\left(
{\begin{array}{*{20}c}
	a^2+\left(b^2-a^2\right)\eta_1^2 & \left(b^2-a^2\right)\eta_1\eta_2\\
	\left(b^2-a^2\right)\eta_1\eta_2 & a^2+\left(b^2-a^2\right)\eta_2^2\\
	\end{array}}
\right)
\end{align*}
with $\eta=\xi/|\xi|\in\mb{S}^1$. Similar as \cite{Reissig2016,Chen2019KV}, we introduce the matrix
\[
\begin{split}
M(\eta):=\left(
{\begin{array}{*{20}c}
	\eta_1 & \eta_2\\
	\eta_2 & -\eta_1 \\
	\end{array}}
\right),
\end{split}
\]
 and define a new variable $W=W(t,\xi)$ such that
\begin{equation*}
W:=\left(
\begin{aligned}
&v_t+i|\xi|\text{diag}(b,a)v\\
&v_t-i|\xi|\text{diag}(b,a)v
\end{aligned}\right),
\end{equation*}
where $v:={M}^{-1}(\eta)\hat{u}$. Moreover, we have $U(t,x)=\ml{F}^{-1}_{\xi\rightarrow x}(W(t,\xi))$. Next, the following first-order system can be derived:
\begin{equation}\label{Eq.FirstOrderSystem}
\left\{\begin{aligned}
&W_t+\left(\frac{1}{2}{B}_0|\xi|^{2\rho}+i{B}_1|\xi|+\frac{1}{2}{B}_0|\xi|^{2\theta}\right)W=0,&&\xi\in\mb{R}^2,\,\,t>0,\\
&W(0,\xi)=W_0(\xi),&&\xi\in\mb{R}^2,
\end{aligned}\right.
\end{equation}
 where the coefficient matrices ${B}_0$ and ${B}_1$ are respectively given by
\begin{equation}\label{Matrix.B0.B1}
\begin{split}
{B}_0=\left(
{\begin{array}{*{20}c}
	1 & 0  & 1 & 0 \\
	0  & 1 &  0 & 1\\
	1 & 0  & 1 & 0 \\
	0 & 1 & 0  & 1\\
	\end{array}}
\right) \,\,\,\, \text{and} \,\,\,\, {B}_1=\left(
{\begin{array}{*{20}c}
	-b & 0  & 0 & 0 \\
	0  & -a &  0 & 0\\
	0 & 0  & b & 0 \\
	0 & 0 & 0  & a\\
	\end{array}}
\right).
\end{split}
\end{equation}
Let us point out that throughout this section, we will study representation of solutions $U=U(t,x)$ to the following Cauchy problem by deriving representation of its partial Fourier transform $W(t,\xi)=\ml{F}_{x\rightarrow \xi}(U(t,x))$:
\begin{equation}\label{Eq.InvFirstOrderSystem}
\left\{\begin{aligned}
&U_t+\frac{1}{2}{B}_0(-\Delta)^{\rho}U+i{B}_1(-\Delta)^{\frac{1}{2}}U+\frac{1}{2}{B}_0(-\Delta)^{\theta}U=0,&&x\in\mb{R}^2,\,\,t>0,\\
&U(0,x)=U_0(x),&&x\in\mb{R}^2,
\end{aligned}\right.
\end{equation}
where the coefficient matrices $B_0$ and $B_1$ are given in \eqref{Matrix.B0.B1}. Moreover, to derive qualitative properties of solutions to \eqref{Eq.DoublyDissElasticWaves}, we only need to study the solutions to \eqref{Eq.InvFirstOrderSystem}.

With the aim of deriving representation of solutions, we may apply WKB analysis and multistep diagonalization procedure (see for example \cite{ReissigWang2005,Yagdjian1997,Jachmann2009,JachmannReissig2008,JachmannReissig2009,Reissig2016,Chen2019TEP}). Before doing these, we should understand the influence of the parameter $|\xi|$ on the asymptotic behavior of solutions to \eqref{Eq.FirstOrderSystem}. Due to our assumption $0\leq 2\rho<1<2\theta\leq2$, we now discuss the influence of $|\xi|$ by three parts. Specifically, we will apply diagonalization procedure for small frequencies $\xi\in\ml{Z}_{\intt}(\varepsilon)$ and large frequencies $\xi\in\ml{Z}_{\extt}(N)$ in Subsections \ref{SubSec.SmallFreq} and \ref{SubSec.LargeFreq}, respectively. Then, the contradiction argument will be applied to prove an exponential stability of solutions for bounded frequencies $\xi\in\ml{Z}_{\text{bdd}}(\varepsilon,N)$ in Subsection \ref{SubSec.BoundedFreq.}.

\subsection{Treatment for small frequencies}\label{SubSec.SmallFreq}
In the case when $\xi\in\ml{Z}_{\intt}(\varepsilon)$, it is clear that the matrix $\frac{1}{2}|\xi|^{2\rho}B_0$ has a dominant influence comparing with the matrices $i|\xi|B_1$ and $\frac{1}{2}|\xi|^{2\theta}B_0$. For this reason, by defining
\begin{align}\label{Matrix.T1}
T_1:=\left(
{\begin{array}{*{20}c}
	-1 & 0  & 1 & 0 \\
	0  & -1 &  0 & 1\\
	1 & 0  & 1 & 0 \\
	0 & 1 & 0  & 1\\
	\end{array}}
\right),
\end{align}
we introduce $W^{(1)}:=T_1^{-1}W$. Then, we may derive
\begin{align*}
W^{(1)}_{t}+\Lambda_1(|\xi|)W^{(1)}+\left(B_0^{(1)}(|\xi|)+B_1^{(1)}(|\xi|)\right)W^{(1)}=0,
\end{align*}
where
\begin{align*}
\Lambda_1(|\xi|)&=|\xi|^{2\rho}\diag(0,0,1,1)=\ml{O}\left(|\xi|^{2\rho}\right),\\
 B_0^{(1)}(|\xi|)&=|\xi|^{2\theta}\diag(0,0,1,1)=\ml{O}\left(|\xi|^{2\theta}\right),\\
B_1^{(1)}(|\xi|)&=i|\xi|T_1^{-1}B_1T_1=\ml{O}(|\xi|).
\end{align*}
Here
\begin{align*}
T_1^{-1}B_1T_1=\left(
{\begin{array}{*{20}c}
	0 & 0  & b & 0 \\
	0  & 0 &  0 & a\\
	b & 0  & 0 & 0 \\
	0 & a & 0  & 0\\
	\end{array}}
\right).
\end{align*}
In the second step we introduce $W^{(2)}:=T_2^{-1}W^{(1)}$, where
\begin{align}\label{Matrix.N2}
T_2:=I_{4\times4}+N_2(|\xi|) \quad\text{with}\quad N_2(|\xi|):=i|\xi|^{1-2\rho}\left(
{\begin{array}{*{20}c}
	0 & 0  & b & 0 \\
	0  & 0 &  0 & a\\
	-b & 0  & 0 & 0 \\
	0 & -a & 0  & 0\\
	\end{array}}
\right).
\end{align}
The following first-order system comes:
\begin{align*}
W_t^{(2)}+\Lambda_1(|\xi|)W^{(2)}+R_2(|\xi|)W^{(2)}=0,
\end{align*}
where
\begin{align*}
R_2=\underbrace{T_2^{-1}B_1^{(1)}(|\xi|)N_2(|\xi|)}_{=\ml{O}(|\xi|^{2-2\rho})}+\underbrace{T_2^{-1}B_0^{(1)}(|\xi|)T_2}_{=\ml{O}(|\xi|^{2\theta})}=\ml{O}\left(|\xi|^{\min\{2-2\rho;2\theta\}}\right).
\end{align*}
To understand the dominant term in the remainder $R_2(|\xi|)$, we distinguish between three cases.
\medskip

\noindent\emph{Case 2.1.1:} $\rho+\theta<1$.\\
In this case the matrix $T_2^{-1}B_0^{(1)}(|\xi|)T_2$ has a dominant influence. We find that this matrix can be rewritten by the following way:
\begin{align*}
T_2^{-1}B_0^{(1)}(|\xi|)T_2=B_0^{(1)}(|\xi|)+T_2^{-1}\left[B_0^{(1)}(|\xi|),N_2(|\xi|)\right].
\end{align*}
Thus, setting $W^{(3)}:=W^{(2)}$ implies
\begin{align*}
W^{(3)}_{t}+(\Lambda_1(|\xi|)+\Lambda_2(|\xi|))W^{(3)}+R_3(|\xi|)W^{(3)}=0,
\end{align*}
where $\Lambda_2(|\xi|)=B_0^{(1)}(|\xi|)=\ml{O}\left(|\xi|^{2\theta}\right)$ and
\begin{align*}
R_3(|\xi|)=\underbrace{T_2^{-1}B_1^{(1)}(|\xi|)N_2(|\xi|)}_{=\ml{O}(|\xi|^{2-2\rho})}+\underbrace{T_2^{-1}[\Lambda_2(|\xi|),N_2(|\xi|)]}_{=\ml{O}(|\xi|^{1+2\theta-2\rho})}=\ml{O}\left(|\xi|^{2-2\rho}\right).
\end{align*}
Because $2\theta>1$, the term $T_2^{-1}B_1^{(1)}(|\xi|)N_2(|\xi|)$ has a dominant influence in comparison with the term $T_2^{-1}[\Lambda_2(|\xi|),N_2(|\xi|)]$ in the remainder $R_3(|\xi|)$. We observe that
\begin{align*}
T_2^{-1}B_1^{(1)}(|\xi|)N_2(|\xi|)&=B_1^{(1)}(|\xi|)N_2(|\xi|)-N_2(|\xi|)T_2^{-1}B_1^{(1)}(|\xi|)N_2(|\xi|),\\
B_1^{(1)}(|\xi|)N_2(|\xi|)&=|\xi|^{2-2\rho}\diag\left(b^2,a^2,-b^2,-a^2\right).
\end{align*}
So, by taking $W^{(4)}:=W^{(3)}$ we have
\begin{align*}
W^{(4)}_{t}+(\Lambda_1(|\xi|)+\Lambda_2(|\xi|)+\Lambda_3(|\xi|))W^{(4)}+R_4(|\xi|)W^{(4)}=0,
\end{align*}
where $\Lambda_3(|\xi|)=B_1^{(1)}(|\xi|)N_2(|\xi|)=\ml{O}\left(|\xi|^{2-2\rho}\right)$ and
\begin{align*}
R_4(|\xi|)=-N_2(|\xi|)T_2^{-1}B_1^{(1)}(|\xi|)N_2(|\xi|)+T_2^{-1}[\Lambda_2(|\xi|),N_2(|\xi|)]=\ml{O}\left(|\xi|^{1+2\theta-2\rho}\right).
\end{align*}
Up to now, we have derived pairwise distinct eigenvalues and $R_4(|\xi|)=\ml{O}\left(|\xi|^{1+2\theta-2\rho}\right)$.
\medskip

\noindent\emph{Case 2.1.2:} $\rho+\theta=1$.\\
In this case the matrices $T_2^{-1}B_0^{(1)}(|\xi|)T_2$ and $T_2^{-1}B_1^{(1)}(|\xi|)N_2(|\xi|)$ have the same influence. For this reason, we set
\begin{align*}
\Lambda_2(|\xi|)=B_1^{(1)}(|\xi|)N_2(|\xi|)+B_0^{(1)}(|\xi|)=|\xi|^{2-2\rho}\diag\left(b^2,a^2,1-b^2,1-a^2\right)=\ml{O}\left(|\xi|^{2-2\rho}\right).
\end{align*}
Then, taking $W^{(3)}:=W^{(2)}$ again we derive
\begin{align*}
W_t^{(3)}+(\Lambda_1(|\xi|)+\Lambda_2(|\xi|))W^{(3)}+R_3(|\xi|)W^{(3)}=0,
\end{align*}
where
\begin{align*}
R_3(|\xi|)=-N_2(|\xi|)T_2^{-1}B_1^{(1)}(|\xi|)N_2(|\xi|)+T_2^{-1}\left[B_0^{(1)}(|\xi|),N_2(|\xi|)\right]=\ml{O}\left(|\xi|^{3-4\rho}\right).
\end{align*}
Up to now, we have derived pairwise distinct eigenvalues and $R_3(|\xi|)=\ml{O}\left(|\xi|^{3-4\rho}\right)$.
\medskip

\noindent\emph{Case 2.1.3:} $\rho+\theta>1$.\\
In this case the matrix $T_2^{-1}B_1^{(1)}(|\xi|)N_2(|\xi|)$ has a dominant influence. Following the idea from \emph{Case 2.1.1} and setting $W^{(3)}:=W^{(2)}$ again, we may derive
\begin{align*}
W_t^{(3)}+(\Lambda_1(|\xi|)+\Lambda_2(|\xi|))W^{(3)}+R_3(|\xi|)W^{(3)}=0,
\end{align*}
where $\Lambda_2(|\xi|)=B_1^{(1)}(|\xi|)N_2(|\xi|)=\ml{O}(|\xi|^{2-2\rho})$ and
\begin{align*}
R_3(|\xi|)=-N_2(|\xi|)T_2^{-1}B_1^{(1)}(|\xi|)N_2(|\xi|)+T_2^{-1}B_0^{(1)}(|\xi|)T_2=\ml{O}\left(|\xi|^{\min\{3-4\rho;2\theta\}}\right).
\end{align*}
Up to now, we have derived pairwise distinct eigenvalues and $R_3(|\xi|)=\ml{O}\left(|\xi|^{\min\{3-4\rho;2\theta\}}\right)$.

Summarizing above diagonalization procedure, according to \cite{Jachmann2009} we obtain the next proposition, which tells us the asymptotic behavior of eigenvalues and representation of solutions.

\begin{prop}\label{Prop.INT}
	The eigenvalues $\lambda_{j}=\lambda_{j}(|\xi|)$ of the coefficient matrix 
	\begin{align*}
	B(|\xi|;\rho,\theta)=\frac{1}{2}\left(|\xi|^{2\rho}+|\xi|^{2\theta}\right)B_0+i|\xi|B_1
	\end{align*}
	 from \eqref{Eq.FirstOrderSystem} behave for $|\xi|<\varepsilon\ll1$ as
	\begin{itemize}
		\item if $\rho+\theta<1$, then
		\begin{equation*}
		\begin{aligned}
		&\lambda_{1}(|\xi|)=b^2|\xi|^{2-2\rho}+\ml{O}\left(|\xi|^{1+2\theta-2\rho}\right),\\
		&\lambda_{2}(|\xi|)=a^2|\xi|^{2-2\rho}+\ml{O}\left(|\xi|^{1+2\theta-2\rho}\right),\\ &\lambda_{3}(|\xi|)=|\xi|^{2\rho}+|\xi|^{2\theta}-b^2|\xi|^{2-2\rho}+\ml{O}\left(|\xi|^{1+2\theta-2\rho}\right),\\
		&\lambda_{4}(|\xi|)=|\xi|^{2\rho}+|\xi|^{2\theta}-a^2|\xi|^{2-2\rho}+\ml{O}\left(|\xi|^{1+2\theta-2\rho}\right);
		\end{aligned}
		\end{equation*}
		\item if $\rho+\theta=1$, then
		\begin{equation*}
		\begin{aligned}
		&\lambda_{1}(|\xi|)=b^2|\xi|^{2-2\rho}+\ml{O}\left(|\xi|^{3-4\rho}\right),\\
		&\lambda_{2}(|\xi|)=a^2|\xi|^{2-2\rho}+\ml{O}\left(|\xi|^{3-4\rho}\right),\\ &\lambda_{3}(|\xi|)=|\xi|^{2\rho}+\left(1-b^2\right)|\xi|^{2-2\rho}+\ml{O}\left(|\xi|^{3-4\rho}\right),\\
		& \lambda_{4}(|\xi|)=|\xi|^{2\rho}+\left(1-a^2\right)|\xi|^{2-2\rho}+\ml{O}\left(|\xi|^{3-4\rho}\right);
		\end{aligned}
		\end{equation*}
		\item if $\rho+\theta>1$, then
		\begin{equation*}
		\begin{aligned}
		&\lambda_{1}(|\xi|)=b^2|\xi|^{2-2\rho}+\ml{O}\left(|\xi|^{\min\{3-4\rho;2\theta\}}\right),\\
		&\lambda_{2}(|\xi|)=a^2|\xi|^{2-2\rho}+\ml{O}\left(|\xi|^{\min\{3-4\rho;2\theta\}}\right),\\
		& \lambda_{3}(|\xi|)=|\xi|^{2\rho}-b^2|\xi|^{2-2\rho}+\ml{O}\left(|\xi|^{\min\{3-4\rho;2\theta\}}\right),\\
		& \lambda_{4}(|\xi|)=|\xi|^{2\rho}-a^2|\xi|^{2-2\rho}+\ml{O}\left(|\xi|^{\min\{3-4\rho;2\theta\}}\right).
		\end{aligned}
		\end{equation*}
		Furthermore, the solution to the Cauchy problem \eqref{Eq.FirstOrderSystem} has in $\ml{Z}_{\intt}(\varepsilon)$ the representation
		\begin{align*}
		W(t,\xi)=T_{\intt}^{-1}(|\xi|)\diag\left(e^{-\lambda_j(|\xi|)t}\right)_{j=1}^4T_{\intt}(|\xi|)W_0(\xi),
		\end{align*}
		where $T_{\intt}(|\xi|)=(I_{4\times4}+N_2(|\xi|))^{-1}T_1^{-1}$ with a matrix $N_2(|\xi|)=\ml{O}\left(|\xi|^{1-2\rho}\right)$ for $|\xi|\rightarrow0$. Here the matrix  $T_1$ is defined in \eqref{Matrix.T1}.
	\end{itemize}
\end{prop}
\subsection{Treatment for large frequencies}\label{SubSec.LargeFreq}
We observe that the symmetric of the system \eqref{Eq.FirstOrderSystem} with respective to the parameters $\rho$ and $\theta$. Thus, by similar procedure we can obtain pairwise distinct eigenvalues. Before stating our result for large frequencies, we define
\begin{align*}
T_3:=I_{4\times4}+N_3(|\xi|) \quad\text{with}\quad N_3(|\xi|)=i|\xi|^{1-2\theta}\left(
{\begin{array}{*{20}c}
	0 & 0  & b & 0 \\
	0  & 0 &  0 & a\\
	-b & 0  & 0 & 0 \\
	0 & -a & 0  & 0\\
	\end{array}}
\right).
\end{align*}

Then, following the similar procedure as the case for small frequencies and according to the thesis \cite{Jachmann2009} we obtain the next proposition.
\begin{prop}\label{Prop.EXT}
	The eigenvalues $\mu_{j}=\mu_{j}(|\xi|)$ of the coefficient matrix
	\begin{align*}
	B(|\xi|;\rho,\theta)=\frac{1}{2}\left(|\xi|^{2\rho}+|\xi|^{2\theta}\right)B_0+i|\xi|B_1
	\end{align*}
	 from \eqref{Eq.FirstOrderSystem} behave for $|\xi|> N\gg1$ as
	\begin{itemize}
		\item if $\rho+\theta<1$, then
		\begin{equation*}
		\begin{aligned}
		&\mu_{1}(|\xi|)=b^2|\xi|^{2-2\theta}+\ml{O}\left(|\xi|^{\min\{3-4\theta;2\rho\}}\right),\\
		&\mu_{2}(|\xi|)=a^2|\xi|^{2-2\theta}+\ml{O}\left(|\xi|^{\min\{3-4\theta;2\rho\}}\right),\\
		& \mu_{3}(|\xi|)=|\xi|^{2\theta}-b^2|\xi|^{2-2\theta}+\ml{O}\left(|\xi|^{\min\{3-4\theta;2\rho\}}\right),\\
		& \mu_{4}(|\xi|)=|\xi|^{2\theta}-a^2|\xi|^{2-2\theta}+\ml{O}\left(|\xi|^{\min\{3-4\theta;2\rho\}}\right).
		\end{aligned}
		\end{equation*}
		\item if $\rho+\theta=1$, then
		\begin{equation*}
		\begin{aligned}
		&\mu_{1}(|\xi|)=b^2|\xi|^{2-2\theta}+\ml{O}\left(|\xi|^{3-4\theta}\right),\\
		&\mu_{2}(|\xi|)=a^2|\xi|^{2-2\theta}+\ml{O}\left(|\xi|^{3-4\theta}\right),\\ &\mu_{3}(|\xi|)=|\xi|^{2\theta}+\left(1-b^2\right)|\xi|^{2-2\theta}+\ml{O}\left(|\xi|^{3-4\theta}\right),\\
		& \mu_{4}(|\xi|)=|\xi|^{2\theta}+\left(1-a^2\right)|\xi|^{2-2\theta}+\ml{O}\left(|\xi|^{3-4\theta}\right);
		\end{aligned}
		\end{equation*}
		\item if $\rho+\theta>1$, then
		\begin{equation*}
		\begin{aligned}
		&\mu_{1}(|\xi|)=b^2|\xi|^{2-2\theta}+\ml{O}\left(|\xi|^{1+2\rho-2\theta}\right),\\
		&\mu_{2}(|\xi|)=a^2|\xi|^{2-2\theta}+\ml{O}\left(|\xi|^{1+2\rho-2\theta}\right),\\ &\mu_{3}(|\xi|)=|\xi|^{2\theta}+|\xi|^{2\rho}-b^2|\xi|^{2-2\theta}+\ml{O}\left(|\xi|^{1+2\rho-2\theta}\right),\\
		& \mu_{4}(|\xi|)=|\xi|^{2\theta}+|\xi|^{2\rho}-a^2|\xi|^{2-2\theta}+\ml{O}\left(|\xi|^{1+2\rho-2\theta}\right);
		\end{aligned}
		\end{equation*}
		Furthermore, the solution to the Cauchy problem \eqref{Eq.FirstOrderSystem} has in $\ml{Z}_{\extt}(N)$ the representation
		\begin{align*}
		W(t,\xi)=T_{\extt}^{-1}(|\xi|)\diag\left(e^{-\mu_j(|\xi|)t}\right)_{j=1}^4T_{\extt}(|\xi|)W_0(\xi),
		\end{align*}
		where $T_{\extt}(|\xi|)=(I_{4\times4}+N_3(|\xi|))^{-1}T_1^{-1}$ with a matrix $N_3(|\xi|)=\ml{O}\left(|\xi|^{1-2\theta}\right)$ for $|\xi|\rightarrow\infty$. Here the matrix  $T_1$ is defined in \eqref{Matrix.T1}.
	\end{itemize}
\end{prop}
\subsection{Treatment for bounded frequencies}\label{SubSec.BoundedFreq.}
Finally, we only need to derive an exponential decay of solutions to \eqref{Eq.FirstOrderSystem} for bounded frequencies to guarantee the exponential stability of solutions.
\begin{prop}\label{Prop.BDD}The solution $W=W(t,\xi)$ to the Cauchy problem \eqref{Eq.FirstOrderSystem} with $0\leq \rho<1/2<\theta\leq1$ fulfills the following exponential decay estimate:
	\begin{align*}
	|W(t,\xi)|\lesssim e^{-ct}|W_0(\xi)|,
	\end{align*}
	for $(t,\xi)\in(0,\infty)\times \ml{Z}_{\text{bdd}}(\varepsilon,N)$, where $c$ is a positive constant.
\end{prop}
\begin{proof}
Let us recall that
\begin{align*}
B(|\xi|;\rho,\theta)=\frac{1}{2}\left(|\xi|^{2\rho}+|\xi|^{2\theta}\right)B_0+i|\xi|B_1.
\end{align*}
It is clear that the eigenvalues of $B(|\xi|;\rho,\theta)$ satisfy
\begin{align*}
0&=\det(B(|\xi|;\rho,\theta)-\lambda I_{4\times 4})\\
&=\lambda^4-2\left(|\xi|^{2\rho}+|\xi|^{2\theta}\right)\lambda^3+\left(\left(|\xi|^{2\rho}+|\xi|^{2\theta}\right)^2+\left(a^2+b^2\right)|\xi|^2\right)\lambda^2\\
&\quad-\left(a^2+b^2\right)|\xi|^2\left(|\xi|^{2\rho}+|\xi|^{2\theta}\right)\lambda+a^2b^2|\xi|^4.
\end{align*}
Now, we assume there exists an eigenvalue $\lambda=id$ with $d\in\mb{R}\backslash\{0\}$. Therefore, the real number $d$ should satisfy the equations
\begin{equation}\label{MIDDLE.EQ}
\left\{
\begin{aligned}
&d^4-\left(\left(|\xi|^{2\rho}+|\xi|^{2\theta}\right)^2+\left(a^2+b^2\right)|\xi|^2\right)d^2+a^2b^2|\xi|^4=0,\\
&id\left(|\xi|^{2\rho}+|\xi|^{2\theta}\right)\left(2d^2-\left(a^2+b^2\right)|\xi|^2\right)=0.
\end{aligned}
\right.
\end{equation}
Due to the facts that $d\neq0$ and $\xi\in \ml{Z}_{\text{bdd}}(\varepsilon,N)$, the equations \eqref{MIDDLE.EQ} leads to
\begin{align*}
-\left(b^2-a^2\right)^2|\xi|^2=2\left(a^2+b^2\right)\left(|\xi|^{2\rho}+|\xi|^{2\theta}\right)^2.
\end{align*}
From our assumption $b>a>0$, we immediately find a contradiction. Thus, there not exists pure imaginary eigenvalue of $B(|\xi|;\rho,\theta)$ for any $0\leq\rho<1/2<\theta\leq 1$ and $\xi\in \ml{Z}_{\text{bdd}}(\varepsilon,N)$. Lastly, by using the compactness of the bounded zone $\ml{Z}_{\text{bdd}}(\varepsilon,N)$ and the continuity of the eigenvalues, the proof is complete.
\end{proof}

\section{Energy estimates}\label{Sec.EnergyEstimate}
The aim of the section is to study the dissipative structure and sharp energy estimates to doubly dissipative elastic waves, where initial data belongs to Bessel potential space with additional $L^m$ regularity ($m\in[1,2]$) or with additional weighted $L^1$ regularity.

The crucial point of sharp energy estimates is to derive the sharp pointwise estimate. By summarizing the results in Propositions \ref{Prop.INT}, \ref{Prop.EXT} and \ref{Prop.BDD}, we obtain the result on the sharp pointwise estimate of solutions to \eqref{Eq.FirstOrderSystem}.
\begin{prop}\label{Prop.PointwiseEst} The solution $W=W(t,\xi)$ to the Cauchy problem \eqref{Eq.FirstOrderSystem} with $0\leq \rho<1/2<\theta\leq1$ satisfies the following pointwise estimates for any $\xi\in\mb{R}^2$ and $t\geq0$:
\begin{align*}
|W(t,\xi)|\lesssim e^{-c\eta(|\xi|)t}|W_0(\xi)|,
\end{align*}
where $\eta(|\xi|):=\frac{|\xi|^{2-2\rho}}{1+|\xi|^{2\theta-2\rho}}$ and $c$ is positive constant.
\end{prop}
\begin{rem}
The pointwise estimate in Proposition \ref{Prop.PointwiseEst} gives the characterization of the dissipative structure of doubly dissipative elastic waves. We now compare the dissipative structure of doubly dissipative elastic waves and elastic waves with friction or structural damping in \cite{Reissig2016,ChenReissig2019SD}. For one thing, as $|\xi|\rightarrow0$, the dissipative structure of doubly dissipative elastic waves is the same as elastic waves with friction or structural damping $(-\Delta)^{\rho}u_t$ for $\rho\in[0,1/2)$, that is $\eta(|\xi|)\asymp|\xi|^{2-2\rho}$ for $|\xi|\rightarrow0$. For another, as $|\xi|\rightarrow\infty$, the dissipative structure of doubly dissipative elastic waves is the same as elastic waves with structural damping $(-\Delta)^{\theta}u_t$ for $\theta\in(1/2,1]$, that is $\eta(|\xi|)\asymp|\xi|^{2-2\theta}$ for $|\xi|\rightarrow\infty$.
\end{rem}

Now, we state our main result on energy estimates.
\begin{thm}\label{Thm.EnergyEst}
Let us consider the Cauchy problem \eqref{Eq.InvFirstOrderSystem} with $0\leq\rho<1/2<\theta\leq1$ and $U_0\in H^s\cap L^m$, where $s\geq0$ and $m\in[1,2]$. Then, the following estimates hold:
\begin{align*}
\|U(t,\cdot)\|_{\dot{H}^s}\lesssim(1+t)^{-\frac{s}{2-2\rho}-\frac{2-m}{m(2-2\rho)}}\|U_0\big\|_{H^s\cap L^m}.
\end{align*} 
\end{thm}
\begin{rem}
According to Proposition \ref{Prop.INT} and sharp pointwise estimate in Proposition \ref{Prop.PointwiseEst}, the energy estimates in Theorem \ref{Thm.EnergyEst} are sharp for initial data $U_0\in H^s\cap L^m$, where $s\geq0$ and $m\in[1,2]$.
\end{rem}
\begin{rem}
We remark that the energy estimates for doubly dissipative elastic waves \eqref{Eq.DoublyDissElasticWaves} in Theorem \ref{Thm.EnergyEst} are the same as damped elastic waves with damping term $(-\Delta)^{\rho}u_t$ for $\rho\in[0,1/2)$ in Theorems 7.2 and 7.3 in \cite{Reissig2016}.
\end{rem}
\begin{rem}\label{Remark.01}
From energy estimates in Theorem \ref{Thm.EnergyEst}, we observe that the decay rate is only determined by the damping term $(-\Delta)^{\rho}u_t$ with $\rho\in[0,1/2)$ in \eqref{Eq.DoublyDissElasticWaves}. For the other damping term $(-\Delta)^{\theta}u_t$ with $\theta\in(1/2,1]$, there is no any influence for the energy estimates. The main reason is that the decay rate for energy estimates of \eqref{Eq.DoublyDissElasticWaves} is mainly determined by dissipative structure for small frequencies. However, for the dissipative structure for small frequencies (see Proposition \ref{Prop.INT}), the dominant influence of eigenvalues are determined by $|\xi|^{2-2\rho}$. Although the parameter $\theta$ in the damping term $(-\Delta)^{\theta}u_t$ has a great influence on the asymptotic behavior of eigenvalues for large frequencies, the solutions satisfies an exponential decay for large frequencies providing that we assume suitable regularity for initial data.
\end{rem}
\begin{proof}
To begin with, by using Proposition \ref{Prop.PointwiseEst}, we calculate
\begin{align*}
\|W(t,\cdot)\|_{\dot{H}^s}&\lesssim\left\|\chi_{\intt}(\xi)|\xi|^se^{-c|\xi|^{2-2\rho}t}W_0(\xi)\right\|_{L^2}+e^{-ct}\left\|\chi_{\text{bdd}}(\xi)|\xi|^sW_0(\xi)\right\|_{L^2}\\
&\quad+\left\|\chi_{\extt}(\xi)|\xi|^se^{-c|\xi|^{2-2\theta}t}W_0(\xi)\right\|_{L^2}\\
&\lesssim\left\|\chi_{\intt}(\xi)|\xi|^se^{-c|\xi|^{2-2\rho}t}W_0(\xi)\right\|_{L^2}+e^{-ct}\|U_0\|_{H^s}.
\end{align*}
Next, we divide the proof into two cases. For the case when $m=2$ in Theorem \ref{Thm.EnergyEst}, we have
\begin{align*}
\left\|\chi_{\intt}(\xi)|\xi|^se^{-c|\xi|^{2-2\rho}t}W_0(\xi)\right\|_{L^2}&\lesssim\sup\limits_{\xi\in\ml{Z}_{\intt}(\varepsilon)}\left||\xi|^se^{-c|\xi|^{2-2\rho}t}\right|\|W_0\|_{L^2}\\
&\lesssim(1+t)^{-\frac{s}{2-2\rho}}\|U_0\|_{L^2}.
\end{align*}
For the case when $m\in[1,2)$ in Theorem \ref{Thm.EnergyEst}, the applications of H\"older's inequality and the Hausdorff-Young inequality yield
\begin{align*}
\left\|\chi_{\intt}(\xi)|\xi|^se^{-c|\xi|^{2-2\rho}t}W_0(\xi)\right\|_{L^2}&\lesssim\left\|\chi_{\intt}(\xi)|\xi|^se^{-c|\xi|^{2-2\rho}t}\right\|_{L^{\frac{2m}{2-m}}}\|U_0\|_{L^m}\\
&\lesssim(1+t)^{-\frac{s}{2-2\rho}-\frac{2-m}{m(2-2\rho)}}\|U_0\|_{L^m}.
\end{align*}
Finally, by applying the Parseval-Plancherel theorem, we immediately complete the proof. 
\end{proof}

Furthermore, we discuss energy estimates in a framework of weighted $L^1$ data.
Before stating our result, we recall the Lemma 2.1 in the paper \cite{Ikehata2004}.
\begin{lem}\label{Lem.IkehataWeightedL1}
	Let $f\in L^{1,\gamma}$ with $\gamma\in(0,1]$. Then, the following estimate holds:
	\begin{align*}
	\left|\hat{f}(\xi)\right|\leq C_{\gamma}|\xi|^{\gamma}\|f\|_{L^{1,\gamma}}+\left|\int_{\mb{R}^n}f(x)dx\right|,
	\end{align*}
	with a positive constant $C_{\gamma}>0$.
\end{lem}

\begin{thm}\label{Thm.EnergyWeighted.L1}
Let us consider the Cauchy problem \eqref{Eq.InvFirstOrderSystem} with $0\leq\rho<1/2<\theta\leq1$ and $U_0\in H^s\cap L^{1,\gamma}$, where $s\geq0$ and $\gamma\in(0,1]$. Then, the following estimates hold:
\begin{align*}
\|U(t,\cdot)\|_{\dot{H}^s}&\lesssim(1+t)^{-\frac{s+\gamma}{2-2\rho}-\frac{1}{2-2\rho}}\|U_0\|_{H^s\cap L^{1,\gamma}}+(1+t)^{-\frac{s}{2-2\rho}-\frac{1}{2-2\rho}}\left|\int_{\mb{R}^2}U_0(x)dx\right|.
\end{align*} 
\end{thm}
\begin{rem}
We remark that if we take initial data satisfying $\left|\int_{\mb{R}^2}U_0(x)dx\right|=0$
in Theorem \ref{Thm.EnergyWeighted.L1}, then we can observe that the decay rates given in Theorem \ref{Thm.EnergyEst} when $m=1$ can be improved by $(1+t)^{-\frac{\gamma}{2-2\rho}}$ for $\gamma\in(0,1]$.
\end{rem}
\begin{proof}
To prove Theorem \ref{Thm.EnergyWeighted.L1}, we only need to modify the estimate for small frequencies. By using Lemma \ref{Lem.IkehataWeightedL1}, we have
\begin{align*}
|W_0(\xi)|\lesssim |\xi|^{\gamma}\|U_0\|_{L^{1,\gamma}}+\left|\int_{\mb{R}^2}U_0(x)dx\right|.
\end{align*}
Then, we derive
\begin{align*}
\left\|\chi_{\intt}(\xi)|\xi|^se^{-c|\xi|^{2-2\rho}t}W_0(\xi)\right\|_{L^2}&\lesssim\left\|\chi_{\intt}(\xi)|\xi|^{s+\gamma}e^{-c|\xi|^{2-2\rho}t}\right\|_{L^2}\|U_0\|_{L^{1,\gamma}}\\
&\quad+\left\|\chi_{\intt}(\xi)|\xi|^{s}e^{-c|\xi|^{2-2\rho}t}\right\|_{L^2}\left|\int_{\mb{R}^2}U_0(x)dx\right|.
\end{align*}
Then, combining with the proof of Theorem \ref{Thm.EnergyEst}, we complete the proof.
\end{proof}
\section{Diffusion phenomena}\label{Sec.AsymptoticProfiles}
Our main purpose in this section is to obtain diffusion phenomena for doubly dissipative elastic waves. According to Theorems \ref{Thm.EnergyEst} and \ref{Thm.EnergyWeighted.L1}, we observe that the decay rate of energy estimates is determined by small frequencies (see Remark \ref{Remark.01}). However, we may obtain an exponential decay estimates with suitable regularity on initial data for bounded frequencies and large frequencies. For this reason, we will interpret diffusion phenomena by the solutions localized in small frequency zone in this section.

 To do this, we first introduce the corresponding reference systems for the cases $\rho+\theta<1$, $\rho+\theta=1$ and $\rho+\theta>1$, respectively. Firstly, we introduce the matrices
\begin{align*}
M_1:=\left(
{\begin{array}{*{20}c}
	b^2 & 0  & 0 & 0 \\
	0  & a^2 &  0 & 0\\
	0 & 0  & -b^2 & 0 \\
	0 & 0 & 0  & -a^2\\
	\end{array}}
\right)\,\,\,\,\text{and}\,\,\,\,M_2:=\left(
{\begin{array}{*{20}c}
	0 & 0  & 0 & 0 \\
	0  & 0 &  0 & 0\\
	0 & 0  & 1 & 0 \\
	0 & 0 & 0  & 1\\
	\end{array}}
\right).
\end{align*}
Motivated by the principle part of eigenvalues in Proposition \ref{Prop.INT}, we define the different reference systems between the following three cases. 
\begin{itemize}
	\item In the case $\rho+\theta<1$, we define $\widetilde{U}=\widetilde{U}(t,x;\rho,\theta)$ is the solution to the following evolution system:
	\begin{equation}\label{Ref.System01}
	\left\{
	\begin{aligned}
	&\widetilde{U}_t+M_1(-\Delta)^{1-\rho}\widetilde{U}+M_2(-\Delta)^{\rho}\widetilde{U}+M_2(-\Delta)^{\theta}\widetilde{U}=0,&&x\in\mb{R}^2,\,\,t>0,\\
	&\widetilde{U}(0,x)=T_1^{-1}U_0(x),&&x\in\mb{R}^2.
	\end{aligned}
	\right.
	\end{equation}
	\item In the case $\rho+\theta=1$, we define $\widetilde{U}=\widetilde{U}(t,x;\rho,\theta)$ is the solution to the following evolution system:
	\begin{equation}\label{Ref.System02}
	\left\{
	\begin{aligned}
	&\widetilde{U}_t+(M_1+M_2)(-\Delta)^{1-\rho}\widetilde{U}+M_2(-\Delta)^{\rho}\widetilde{U}=0,&&x\in\mb{R}^2,\,\,t>0,\\
	&\widetilde{U}(0,x)=T_1^{-1}U_0(x),&&x\in\mb{R}^2.
	\end{aligned}
	\right.
	\end{equation}
	\item In the case $\rho+\theta>1$, we define $\widetilde{U}=\widetilde{U}(t,x;\rho,\theta)$ is the solution to the following evolution system:
	\begin{equation}\label{Ref.System03}
	\left\{
	\begin{aligned}
	&\widetilde{U}_t+M_1(-\Delta)^{1-\rho}\widetilde{U}+M_2(-\Delta)^{\rho}\widetilde{U}=0,&&x\in\mb{R}^2,\,\,t>0,\\
	&\widetilde{U}(0,x)=T_1^{-1}U_0(x),&&x\in\mb{R}^2.
	\end{aligned}
	\right.
	\end{equation}
\end{itemize}
Here the matrix $T_1$ is defined in \eqref{Matrix.T1}.

Let us now give some explanation for these reference system.

In the case when $\rho+\theta<1$, for the evolution system \eqref{Ref.System01}, we find that the reference system is made up of three different parabolic systems. We may interpret this new effect as \emph{triple diffusion phenomena}. This effect is shown firstly in \cite{Chen2019TEP} for thermoelastic plate equations with structural damping. In this case, the damping term $(-\Delta)^{\theta}u_t$ with $\theta\in(1/2,1]$ in \eqref{Eq.DoublyDissElasticWaves} really has influence on the diffusion structure. But this effect does not appear in the other case $\rho+\theta\geq1$.

However, we find that when $\rho+\theta\geq1$, the reference system \eqref{Ref.System01} is changed into \eqref{Ref.System02} and \eqref{Ref.System03}. Obviously, these reference systems are only made up of two different parabolic systems, whose structures are similar as reference system for elastic waves with damping term $(-\Delta)^{\rho}u$ for $\rho\in[0,1/2)$. We may interpret this effect as \emph{double diffusion phenomena} (one may see the pioneering paper \cite{DabbiccoEbert2014}).

From the above discussions, we observe a new threshold of diffusion structure for doubly dissipative elastic waves, that is $\rho+\theta=1$. In other words, the structure of the reference system will be changed with the parameters changing from $\rho+\theta<1$ to $\rho+\theta\geq1$.

Let us begin to state our main theorems on diffusion phenomena.

\begin{thm}\label{Thm.DiffusionPhenomena}
Let us consider the Cauchy problem \eqref{Eq.InvFirstOrderSystem} with $0\leq \rho<1/2<\theta\leq1$ and $U_0\in L^m$ with $m\in[1,2]$. Then, the following refinement estimates hold:
\begin{align*}
&\left\|\chi_{\intt}(D)\left(U(t,\cdot)-T_1\widetilde{U}(t,\cdot;\rho,\theta)\right)\right\|_{\dot{H}^s}\lesssim(1+t)^{-\frac{s}{2-2\rho}-\frac{2-m}{m(2-2\rho)}-q(\rho,\theta)}\|U_0\|_{L^m},
\end{align*}
where the function $q=q(\rho,\theta)$ is defined by 
\begin{equation}\label{Fun.q(rho,theta)}
q(\rho,\theta):=
\left\{
\begin{aligned}
&\frac{2\theta-1}{2-2\rho},&&\text{if }\rho+\theta<1,\\
&\frac{1-2\rho}{2-2\rho},&&\text{if }\rho+\theta\geq1,\\
\end{aligned}
\right.
\end{equation}
the matrix $T_1$ is defined in \eqref{Matrix.T1}.
\end{thm}
\begin{proof}
	Here we only prove the case when $\rho+\theta<1$. For the other case when $\rho+\theta\geq1$, its proof is similar as the following discussion. Thus, we omit it.

    First of all, let us apply the partial Fourier transform with respect to spatial variable such that $\widetilde{W}(t,\xi;\rho,\theta)=\ml{F}_{x\rightarrow\xi}(\widetilde{U}(t,x;\rho,\theta))$ to get
	\begin{align*}
	\widetilde{W}(t,\xi;\rho,\theta)=\diag\left(e^{-\tilde{\lambda}_j(|\xi|)t}\right)_{j=1}^4T_1^{-1}W_0(\xi),
	\end{align*}
	where
	\begin{equation*}
	\begin{aligned}
	&\tilde{\lambda}_{1}(|\xi|)=b^2|\xi|^{2-2\rho},&&\tilde{\lambda}_{2}(|\xi|)=a^2|\xi|^{2-2\rho},\\ &\tilde{\lambda}_{3}(|\xi|)=|\xi|^{2\rho}+|\xi|^{2\theta}-b^2|\xi|^{2-2\rho},&& \tilde{\lambda}_{4}(|\xi|)=|\xi|^{2\rho}+|\xi|^{2\theta}-a^2|\xi|^{2-2\rho}.
	\end{aligned}
	\end{equation*}
	We remark that $\tilde{\lambda}_j(|\xi|)$ are the principle parts of eigenvalues $\lambda_j(|\xi|)$ for $j=1,\dots,4$ (one may recall the statement of Proposition \ref{Prop.INT}).\\
According to the representation of solutions for $\xi\in\ml{Z}_{\intt}(\varepsilon)$ in Proposition \ref{Prop.INT}, we may obtain
\begin{align*}
\chi_{\intt}(\xi)|\xi|^s\left(W(t,\xi)-T_1\widetilde{W}(t,\xi;\rho,\theta)\right)=\chi_{\intt}(\xi)|\xi|^s\left(J_1(t,|\xi|)+J_2(t,|\xi|)+J_3(t,|\xi|)\right)W_0(\xi),
\end{align*}
where
\begin{align*}
J_1(t,|\xi|)&=T_1\diag\left(e^{-\lambda_j(|\xi|)t}-e^{-\tilde{\lambda}_j(|\xi|)t}\right)_{j=1}^4T_1^{-1},\\
J_2(t,|\xi|)&=T_1N_2(|\xi|)\diag\left(e^{-\lambda_j(|\xi|)t}\right)_{j=1}^4(I_{4\times4}+N_2(|\xi|))^{-1}T_1^{-1},\\
J_3(t,|\xi|)&=-T_1(I_{4\times4}+N_2(|\xi|))\diag\left(e^{-\lambda_j(|\xi|)t}\right)_{j=1}^4N_2(|\xi|)(I_{4\times 4}+N_2(|\xi|))^{-1}T_1^{-1}.
\end{align*}
Here the matrix $N_2(|\xi|)=\ml{O}\left(|\xi|^{1-2\rho}\right)$ is defined in \eqref{Matrix.N2}. In the above equation we used
\begin{align*}
(I_{t\times4}+N_2(|\xi|))^{-1}=I_{4\times4}-N_2(|\xi|)(I_{4\times4}+N_2(|\xi|))^{-1}.
\end{align*}

We now begin to estimate $J_1(t,|\xi|)$ and $J_2(t,|\xi|)+J_3(t,|\xi|)$, respectively. By means value theorem, we know that
\begin{align*}
\chi_{\intt}(\xi)\left|e^{-\lambda_j(|\xi|)t}-e^{-\tilde{\lambda}_j(|\xi|)t}\right|\lesssim (1+t)\chi_{\intt}(\xi)|\xi|^{1+2\theta-2\rho}e^{-\tilde{\lambda}_j(|\xi|)t}.
\end{align*}
Thus,
\begin{align*}
\left\|\chi_{\intt}(\xi)|\xi|^{s}J_1(t,\xi)W_0(\xi)\right\|_{L^2}&\lesssim(1+t)\left\|\chi_{\intt}(\xi)|\xi|^{s+1+2\theta-2\rho}\diag\left(e^{-\tilde{\lambda}_j(|\xi|)t}\right)_{j=1}^4W_0(\xi)\right\|_{L^2}\\
&\lesssim(1+t)^{-\frac{s}{2-2\rho}-\frac{2-m}{m(2-2\rho)}-\frac{2\theta-1}{2-2\rho}}\|U_0\|_{L^m}.
\end{align*}
Due to the fact that $N_2(|\xi|)=\ml{O}\left(|\xi|^{1-2\rho}\right)$, we have
\begin{align*}
\left\|\chi_{\intt}(\xi)|\xi|^{s}(J_2(t,\xi)+J_3(t,|\xi|))W_0(\xi)\right\|_{L^2}&\lesssim\left\|\chi_{\intt}(\xi)|\xi|^{s+1-2\rho}\diag\left(e^{-\lambda_j(|\xi|)t}\right)_{j=1}^4W_0(\xi)\right\|_{L^2}\\
&\lesssim(1+t)^{-\frac{s}{2-2\rho}-\frac{2-m}{m(2-2\rho)}-\frac{1-2\rho}{2-2\rho}}\|U_0\|_{L^m}
\end{align*}
Summarizing the above estimates leads to
\begin{align*}
\left\|\chi_{\intt}(\xi)|\xi|^s\left(W(t,\xi)-T_1\widetilde{W}(t,\xi;\rho,\theta)\right)\right\|_{L^2}\lesssim(1+t)^{-\frac{s}{2-2\rho}-\frac{2-m}{m(2-2\rho)}-\frac{2\theta-1}{2-2\rho}}\|U_0\|_{L^m},
\end{align*}
where we used our condition $\rho+\theta<1$.\\
Finally, applying the Parseval-Plancherel theorem, we complete the proof of the theorem.
\end{proof}
\begin{thm}\label{Thm.DiffusionPhenWeightL1}
Let us consider the Cauchy problem \eqref{Eq.InvFirstOrderSystem} with $0\leq \rho<1/2<\theta\leq1$ and $U_0\in L^{1,\gamma}$ with $\gamma\in(0,1]$. Then, the following refinement estimates hold:
\begin{align*}
&\left\|\chi_{\intt}(D)\left(U(t,\cdot)-T_1\widetilde{U}(t,\cdot;\rho,\theta)\right)\right\|_{\dot{H}^s}\lesssim(1+t)^{-\frac{s+\gamma}{2-2\rho}-\frac{1}{2-2\rho}-q(\rho,\theta)}\|U_0\|_{L^{1,\gamma}},
\end{align*}
where the function $q=q(\rho,\theta)$ is defined in \eqref{Fun.q(rho,theta)} and the matrix $T_1$ is defined in \eqref{Matrix.T1}.
\end{thm}
\begin{proof}
We may immediately compete the proof of this result by following the procedure from the proofs of Theorems \ref{Thm.EnergyWeighted.L1} and \ref{Thm.DiffusionPhenomena}.
\end{proof}

\begin{rem}
According to Theorems \ref{Thm.EnergyEst}, \ref{Thm.EnergyWeighted.L1}, \ref{Thm.DiffusionPhenomena} and \ref{Thm.DiffusionPhenWeightL1}, the decay rate $(1+t)^{-q(\rho,\theta)}$ can be gained by subtracting the solutions $\widetilde{U}(t,x;\rho,\theta)$ for the reference systems \eqref{Ref.System01}, \eqref{Ref.System02} and \eqref{Ref.System03}. From the value of $q(\rho,\theta)$, we also find that the threshold for diffusion structure is $\rho+\theta=1$.
\end{rem}

\section{Concluding remarks}\label{Sec.ConcludingRemarks}
\begin{rem}
Let us discuss about smoothing effect of solutions. We first introduce the Gevrey space $\Gamma^{\kappa}$ with $\kappa\in[1,\infty)$ (see \cite{Rodino1993}), where
\begin{align*}
\Gamma^{\kappa}:=\left\{f\in L^2\,:\,\text{there exists a constant }c\text{ such that }\exp\left(c\langle\xi\rangle^{\frac{1}{\kappa}}\right)\ml{F}(f)\in L^2\right\}.
\end{align*}
By using Proposition \ref{Prop.EXT} with the same approach of \cite{Reissig2016}, we immediately obtain the following results.
\begin{thm}
Let us consider the Cauchy problem \eqref{Eq.InvFirstOrderSystem} with $0\leq\rho<1/2<\theta<1$ and $U_0\in L^2$. Then, the solutions satisfy $|D|^sU(t,\cdot)\in\Gamma^{\frac{1}{2-2\theta}}$ with $s\geq0$. However, when $0\leq\rho<1/2<\theta=1$, the solutions do not belong to any Gevrey space.
\end{thm}
\noindent It is well-known that smoothing effect is mainly determined by asymptotic behavior of eigenvalues localized in large frequency zone (see Proposition \ref{Prop.EXT}). For this reason, we may observe smoothing effect is only influenced by the damping term $(-\Delta)^{\theta}u_t$ with $\theta\in(1/2,1]$ in the Cauchy problem \eqref{Eq.DoublyDissElasticWaves}. 
\end{rem}
\begin{rem}
From \cite{Reissig2016,ChenReissig2019SD}, we know the solution to elastic waves with friction $u_t$ does not have smoothing effect. However, in doubly dissipative elastic waves \eqref{Eq.DoublyDissElasticWaves}, the structural damping $(-\Delta)^{\theta}u_t$ with $\theta\in(1/2,1)$ brings Gevrey smoothing for the solutions even when $\rho=0$.
\end{rem}
\begin{rem}
In the present paper we  focus on energy estimates with initial data taking from $H^s\cap L^m$ for $s\geq0$, $m\in[1,2]$ or from $H^{s}\cap L^{1,\gamma}$ for $\gamma\in(0,1]$. Here we restrict ourselves on estimating solutions in the $L^2$ norm. For estimating the solutions in the $L^q$ norm with $2\leq q\leq\infty$, by applying Lemma 4.2 in \cite{Chen2019TEP}, one may obtain $L^p-L^q$ estimates with $1\leq p\leq 2\leq q\leq\infty$ and diffusion phenomena in a $L^p-L^q$ framework. 
\end{rem}
\begin{rem}
Our aim in this paper is to investigate dissipative structure and diffusion phenomena for doubly dissipative elastic waves \eqref{Eq.DoublyDissElasticWaves} in two spaces dimensions, especially, we obtain a new threshold for diffusion structure. We think it is also possible to study three dimensional doubly dissipative elastic waves without any new difficulties. The crucial point is to derive asymptotic behavior of eigenvalues and representation of solutions by using suitable diagonalization procedure.
\end{rem}
\section*{Acknowledgments} 

The PhD study of the author is supported by S\"achsiches Landesgraduiertenstipendium.

\bibliographystyle{elsarticle-num}

\end{document}